\documentclass[conference]{IEEEtran}
\pdfoutput=1

\usepackage{calrsfs}
\DeclareMathAlphabet{\pazocal}{OMS}{zplm}{m}{n}
\usepackage[dvips]{color}
\usepackage{times}
\usepackage{graphicx}
\usepackage{epstopdf}
\usepackage{subcaption}
\usepackage{amsmath}
\usepackage{amssymb}
\usepackage{amsxtra}
\usepackage{here}
\usepackage{times}
\usepackage{url}
\usepackage{cite}
\usepackage[ruled,vlined]{algorithm2e}
\usepackage{amsmath,amsthm}
\usepackage{dsfont}
\usepackage{verbatim} 
\usepackage{lettrine}

\newtheorem{theorem}{\bf Theorem}

\newtheorem{proposition}{\bf Proposition}
\newtheorem{lemma}{\bf Lemma}

\makeatletter
	
\IEEEoverridecommandlockouts

\begin{document}
\title{Resilient Critical Infrastructure: Bayesian Network Analysis and Contract-Based Optimization}   

\author{\IEEEauthorblockN{AbdelRahman Eldosouky$^1$, Walid Saad$^1$, and Narayan Mandayam$^2$}\\ 
\IEEEauthorblockA{\small $^1$ Wireless@VT, Bradley Department of Electrical and Computer Engineering, Virginia Tech, Blacksburg, VA, USA, \\Emails:\{iv727,walids\}@vt.edu\\
$^2$ Electrical and Computer Engineering Department , Rutgers University, New Brunswick, NJ, USA, Email: narayan@winlab.rutgers.edu\\
}
\thanks{This research was supported by the US National Science Foundation under Grants ACI-1541105 and ACI-1541069.} \vspace{-0.9cm}
}
\date{}
\maketitle

\begin{abstract}
Instilling resilience in critical infrastructure (CI) such as dams or power grids is a major challenge for tomorrow's cities and communities. Resilience, here, pertains to a CI's ability to adapt or rapidly recover from disruptive events.
In this paper, the problem of optimizing and managing the resilience of CIs is studied.
In particular, a comprehensive two-fold framework is proposed to improve CI resilience by considering both the individual CIs and their collective contribution to an entire system of multiple CIs. To this end, a novel analytical resilience index is proposed to measure the effect of each CI's physical components on its probability of failure.
In particular, a Markov chain defining each CI's performance state and a Bayesian network modeling the probability of failure are introduced to infer each CI's resilience index.
Then, to maximize the resilience of a system of CIs, a novel approach for allocating resources, such as drones or maintenance personnel, is proposed.
In particular, a comprehensive resource allocation framework, based on the tools of contract theory, is proposed enabling the system operator to optimally allocate resources, such as, redundant components or monitoring devices to each individual CI based on its economic contribution to the entire system.
The optimal solution of the contract-based resilience resource allocation problem is analytically derived using dynamic programming.
The proposed framework is then evaluated using a case study pertaining to hydropower dams and their interdependence to the power grid.
Simulation results, within the case study, show that the system operator can economically benefit from allocating the resources while dams have a $60\%$ average improvement over their initial resilience indices.
\end{abstract}
\vspace{0.2cm}


\section{Introduction}\vspace{-0.1cm}
\IEEEPARstart{C}{ritical} infrastructure (CI), such as power grids and transportation systems, are vital to modern day cities and communities~\cite{ICIP28}. As such, maintaining proper operation of CIs, in presence of failures or security threats, is therefore a critical challenge.
In particular, \emph{reliability and resilience} are two key measures that can be used to evaluate the functionality and the ability of an infrastructure to deliver its designated service, under potentially disruptive situations.
In practice, there is a significant difference between reliability and resilience. Reliability is a term that describes the frequency or the likelihood of a CI's failure~\cite{ICIP13}. Resilience, on the other hand, has multiple definitions that are typically application-dependent~\cite{ICIP28}. Most of these definitions pertain to resilience in response to a change in or a corruption to the system's normal functionality.
A general definition of resilience, given by the Department of Homeland Security (DHS) advisory council, is the ability of an infrastructure to adapt to or rapidly recover from a potentially disruptive event~\cite{ICIP19}.

The importance of studying reliability and resilience for CIs stems from the fact that they are prone to many disruptive events such as natural disasters, hazardous conditions, subversive attacks, aging, or even inadequate maintenance~\cite{ICIP37}.
Thus, it is crucial for CIs to operate reliably and to be resilient in face of potential failures and disruptions.
Given that CIs cut across multiple domains that include communications, dams, power grids, transportation systems, and water systems~\cite{CIP01}, and that the resilience lacks a standard definition, resilience improvement techniques are typically infrastructure-specific, for instance ~\cite{ICIP13} considered the resilience of water systems,~\cite{ICIP42} proposed a framework to improve the resilience of the power grid, and~\cite{ICIP40} considered the resilience of petrochemical CIs.
This poses many challenges for assessing and developing resilience improvement techniques for different-type interdependent CIs.
A general framework is therefore needed to evaluate the resilience of different CIs and to help in designing general resilience improvement techniques.
Some studies in the literature, e.g., ~\cite{ICIP31,ICIP39,ICIP34,ICIP20}, proposed general resilience frameworks for CIs.
However, the approaches proposed in this prior art mostly evaluate the CI resilience based on satisfying a number of pre-determined criteria as detailed in the next section.
The resilience measures based on these properties fail to capture the effect of different disruptive events on the CI. In contrast, here, our goal is to introduce a general framework to evaluate and improve CI resilience based on the effect of disruptive events on the CI's components.
Prior to providing our key contributions, we will first review existing related frameworks and techniques in the next section to pinpoint their limitations.

\subsection{Related Work}
Critical Infrastructure resilience has recently attracted significant attention~\cite{ICIP31,ICIP39,ICIP34,ICIP20,ICIP41}.
In~\cite{ICIP31} the authors considered four properties for resilience: robustness, redundancy, resourcefulness, and rapidity and the resilience was quantified using four interrelated dimensions: technical (physical), organizational, social, and economic.
The authors in~\cite{ICIP39} proposed a resilience framework that seeks to achieve three resilience properties pertaining to the ability of a system to absorb the impacts of perturbations, adapt to undesirable situations, and quickly return to its normal operations.
In~\cite{ICIP34}, a three-stage framework, reflecting the infrastructure's resistant, absorptive, and restorative capacities, is introduced to analyze the resilience.
The DHS work in~\cite{ICIP20} developed the notion of a resilience measurement index (RMI) which is an indicator to determine the degree to which the elements pertaining to resilience have been implemented by a CI. These elements include the preparedness of the CI to possible failures and the extent to which recovery mechanisms and mitigation measures are installed.
The work in ~\cite{ICIP41} introduced a quantitative assessment for infrastructure's resilience using optimal control design in which recovery processes and costs are integrated to derive the resilience.
However, one key limitation of these studies, ~\cite{ICIP31,ICIP39,ICIP34,ICIP20,ICIP41}, is that they can be used to compare different CIs, yet, they do not capture the effect of specific events on the infrastructure. Therefore, their use is mostly limited to evaluating the resilience of CI but not to improving it.

Other studies in the literature have focused on improving CI resilience by allocating CI-specific physical resources ~\cite{ICIP33,ICIP29,ICIP32,ICIP24}.
In~\cite{ICIP33}, the resilience of a cyber-physical system is improved by allocating a number of inter-network edges to the nodes of the interdependent network connecting the system's cyber and physical layers. The effect of cascading failure among nodes is studied to help in the process of resource allocation.
The authors in~\cite{ICIP29} proposed a new approach to repair system components using a graph-theoretic approach.
In ~\cite{ICIP32} and \cite{ICIP24}, CI resilience is studied from a general perspective without defining a quantitative metric for resilience.
The authors in ~\cite{ICIP32} consider the problem of allocating resources to highway bridges to improve the resilience of a transportation system.
In~\cite{ICIP24}, a framework is proposed to allocate resources to CIs based on their vulnerability level. Contract theory is used to formulate the problem to optimize the economic benefit from the allocated resources which are offered to CIs through contracts managed by the system operator.
Note that, in ~\cite{ICIP34}, beyond defining resilience properties, a framework is proposed to improve CI resilience. The framework depends on allocating resources to improve the resilience by hardening CI's components, duplicating components, or ensuring rapid recovery of failed components.
The framework is applied to improve the resilience of a power grid whose components are the generators and the resources are allocated to the generators.

One limitation of these previous studies ~\cite{ICIP34} and ~\cite{ICIP33,ICIP29,ICIP32,ICIP24}, is that individual CIs are abstracted within the system, e.g. as nodes within a generic graph.
This provides no information on improving individual CIs resilience as the solutions introduced in these studies ~\cite{ICIP34} and ~\cite{ICIP33,ICIP29,ICIP32,ICIP24} consider the resilience of an entire system of multiple CIs while being agnostic to each individual CI's resilience properties.
Indeed, individual CIs and their specific failures are largely abstracted and not considered in enough details. Hence, in such prior art, when resources are allocated within the system, no information is provided on how to effectively allocate them at the level of each CI.

In light of the preceding discussions, we propose a general framework to study and improve the resilience of CIs. The framework addresses \emph{the resilience at the level of both individual CIs and their collective effect on an entire system of multiple CIs}. We introduce an analytical resilience index to quantify the resilience of individual infrastructures and to give insights about improving this resilience. Resilience is evaluated as a function of the CI's probability of failure derived from the cascading failure of its physical components. Resources are then allocated to the individual CIs according to their contribution to the entire system.
Examples of resources here include redundant components or monitoring devices such as sensors or cameras.
Finally, each infrastructure can use the allocated resources to improve its resilience based on the introduced allocation algorithm.  The key contributions stemming from this framework are outlined next.

\subsection{Contributions}
The main contribution of this paper is a comprehensive analytical framework for analyzing and optimizing the resilience of CIs. The proposed framework can be applied to different systems and CIs to evaluate their resilience and optimize it. The framework considers the resilience of each individual CI and allows improving it based on the economic contribution of each CI to the entire system of multiple CIs.
We model the CI performance state using a Markov chain that allows us to derive a novel quantifiable resilience index.
The proposed resilience index relates to the CI's probability of failure which is induced from the probabilistic inference of a Bayesian network modeling the relationships between the various components of a given CI.
The Bayesian network captures the effect of each component's failure on the CI's probability of failure. This allows calculating the effect of fixing each component on the probability of failure and hence on the infrastructure's resilience index. 
We also develop an algorithm, using the Bayesian network, to prioritize each CI's components based on their effect on the resilience index. This algorithm can be used by individual CIs to determine the order in which they should secure their key components through external resources.

A case study pertaining to hydropower dams is introduced to highlight the importance of the proposed framework and to evaluate its performance. Within this case study, a hydropower dam's resilience is evaluated based on its probability of successfully generating electricity. We propose a new approach for improving the dam's resilience by securing its main components using external resources. The problem of allocating these resources to multiple dams, based on their economic contribution to the entire system (the power grid), is modeled using \emph{contract theory}~\cite{CT00} and the optimal solution to this problem is derived using dynamic programming.
Through simulations, we show that both the system operator and individual CIs can benefit from the process of resource allocation. The system operator can maximize its reward from the allocated resources using contract theory, while CIs significantly improve their resilience indices.

\begin{figure}[t]
  \centerline{\includegraphics[width=8cm]{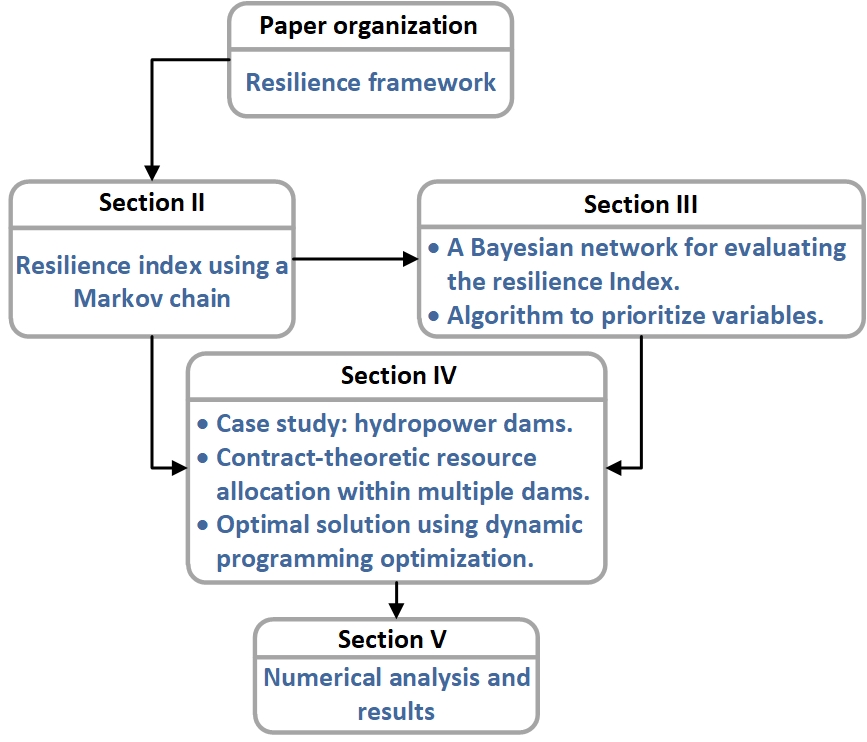}}
  \caption{Organization of the paper.}\label{fig:2}
  \vspace{-.2cm}
\end{figure}

The rest of this paper is organized as summarized in Fig. ~\ref{fig:2}. The Markov chain model and the analytical analysis for deriving the resilience index is presented in Section~\ref{Sec:Model}. The Bayesian network analysis and components' prioritization algorithm is discussed in Section~\ref{Sec:Bayesian}. The case study of hydropower dams and the optimal solution to the problem of CI resource allocation is derived in Section~\ref{Sec:Resources}. Numerical results are presented and analyzed in Section~\ref{Sec:results}. Finally, conclusions are drawn in Section~\ref{Sec:conclusion}.

\section{Evaluating the Resilience of Critical Infrastructure using Markov Chains}\label{Sec:Model}

Consider a critical infrastructure whose performance at a given time $n$ is a function of the state of the system as captured by the random variable $Y_n$. $Y_n$ can take values from a set $\{S,W,F\}$ whose values represent three CI states: success ($S$), warning ($W$), and failure ($F$). The success state, $S$, represents normal service, i.e., the infrastructure is properly delivering its designated service. The CI will be in a warning state, $W$,  with the occurrence of a partial failure to its components that may lead to a complete failure. The failure state, $F$, represents the failure of the CI to deliver its designated service. This failure can occur either suddenly due to, e.g., natural disasters, or as a result of a partial failure from a previous warning state.
We introduce a Markov chain to model these states as shown in Fig.~\ref{fig:1}.

\begin{figure}[t]
  \centerline{\includegraphics[width=7cm]{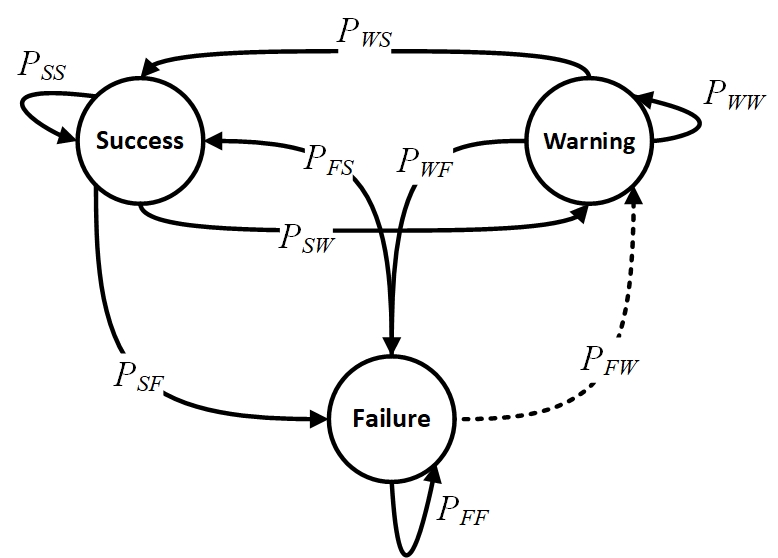}}
  \caption{Markov chain modeling the states of a CI.}\label{fig:1}
    \vspace{-.2cm}
\end{figure}

The transition probabilities between the different states can be induced from the Markov chain. Let $P_{AB}=\textrm{Prob} [Y_{n+1} = B | Y_n = A]$ where $A,B \in \left\{S,W,F \right\}$. Then, the full transition probability matrix $\boldsymbol{P}$ will be given by:
\begin{equation}
\boldsymbol{P} =
  \begin{bmatrix}
    P_{SS} & P_{SW} & P_{SF} \\
    P_{WS} & P_{WW} & P_{WF} \\
	P_{FS} & P_{FW} & P_{FF}
  \end{bmatrix}.
\end{equation}

The values within each row of $\boldsymbol{P}$ will sum to $1$ as they represent the probability distribution for all possible next states whenever the system is at a specific state.

In practice, it is intuitive to assume that the probability $P_{FW}$ is zero. This is due to the fact that, once a failure happens, the CI will either remain at this state or will recover to the success state not to a warning state. Here, we also assume that, whenever the infrastructure is at a warning state, it is either fixed and restored to a success state or it continues to fail and eventually goes to a failure state.
This transition is based on the actions taken at a given time step, however, there is still a small probability that no action is taken at this time step as captured by the probability of remaining in the warning state $P_{WW}$. We assume that $P_{WW}$ is fixed to a value $\epsilon$ which should be small. Based on these assumptions, the transition probability can be simplified, as follows:
\begin{equation}
\boldsymbol{P} =
  \begin{bmatrix}
    P_{SS} & P_{SW} & P_{SF} \\
    P_{WS} & \epsilon & 1- \epsilon - P_{WS} \\
	P_{FS} & 0 & 1-P_{FS}
  \end{bmatrix}.
\end{equation}

Note that, the probability matrix $\boldsymbol{P}$ specifies transition probabilities for a single time step. Transition probabilities for $n$ time steps can be calculated as $\boldsymbol{P}^n$. 
The probabilities of being at a given state i.e., $P(S),P(W)$, and $P(F)$, can be calculated using $\boldsymbol{P}^n$ and the vector of the initial probability distribution of being at each state $\boldsymbol{U}_{0}$, as follows:
\begin{equation}\label{eq:Prob_n}
\boldsymbol{U}_{n} = \boldsymbol{U}_{0} \cdot \boldsymbol{P}^n,
\end{equation}
where $\boldsymbol{U}_{n}=[u^S_n,u^W_n,u^F_n]$ is the probability distribution vector of being at each state after $n$ time steps.

We define the resilience $\gamma$ of an infrastructure to be the reciprocal of the probability of being in the failure state in the long run as:
\begin{equation}
\gamma = \lim_{n \to \infty} \frac{1}{u^F_n},
\end{equation}
where ${u}^F_{n}$ is the probability corresponding to the failure state in the vector $\boldsymbol{U}_n$. This probability will always exist if the Markov chain is irreducible as shown in ~\cite{grinstead2012introduction}, i.e., there is no absorbing state.
In practice, the chain in Fig.~\ref{fig:1} cannot have an absorbing state so it is irreducible. 
The resilience here, is inversely proportional to the probability of failure thus reducing the probability of failure will increase the resilience. Reducing the long run probability of failure implies that the CI will have fewer time steps at the failure state which conforms to the definition in ~\cite{ICIP19}.
Note that, relating the resilience to the probability of failure was introduced in ~\cite{ICIP13} however, in ~\cite{ICIP13}, the CI was allowed to only operate in either a satisfactory or a failure state. 

To calculate the value of $\gamma$, the value of $u^F_n$ needs to be evaluated at high values of $n$ which in turn will depend on $\boldsymbol{P}^n$.
The powers of $\boldsymbol{P}$ for high values of $n$ can be calculated in advance if $\boldsymbol{P}$ is a regular transition matrix ~\cite{grinstead2012introduction}. The powers $\boldsymbol{P}^n$ are shown to converge to a matrix $\boldsymbol{V}$ in which all rows are the same and each row is a strictly positive probability distribution vector ~\cite{grinstead2012introduction} if $\boldsymbol{P}$ is a regular transition matrix. Converging to a constant matrix means any further multiplications of $\boldsymbol{V}$ with $\boldsymbol{P}$ will not change $\boldsymbol{V}$, i.e.,
\begin{equation}\label{eq:converge}
\boldsymbol{V} \cdot \boldsymbol{P} = \boldsymbol{V}.
\end{equation}

As all rows in $\boldsymbol{V}$ are the same, the probability vector $\boldsymbol{U}_{n}$, in (\ref{eq:Prob_n}), will no longer depend on the initial probability distribution $\boldsymbol{U}_{0}$. Multiplying the values of  $\boldsymbol{U}_{0}$, which sum up to $1$, with the constant columns of $\boldsymbol{V}$ will yield the same constant values of $\boldsymbol{V}$. Hence, (\ref{eq:Prob_n}) can then be written as:
\begin{equation}
\boldsymbol{U}_{n} = \boldsymbol{U}_{0} \cdot \boldsymbol{V}= \boldsymbol{V}.
\end{equation}

Let the vector $\boldsymbol{v}$ be the constant row of $\boldsymbol{V}$ which assigns different transition probabilities to all possible states. According to (\ref{eq:converge}), this vector will satisfy the following property:
\begin{equation}\label{regular:eq}
\boldsymbol{v} \cdot \boldsymbol{P} = \boldsymbol{v}.
\end{equation}

As discussed earlier, matrix $\boldsymbol{V}$ only exists if matrix $\boldsymbol{P}$ is a regular Markov matrix. In the following theorem, we prove the necessary conditions that $\boldsymbol{P}$ must satisfy in order to be a regular Markov matrix.

\begin{theorem}\label{theo1}
The probability transition matrix $\boldsymbol{P}$ is a regular Markov matrix.
\end{theorem}

\begin{proof} 
Let $\boldsymbol{v}=[v^S,v^W,v^F]$, then, the values of $v^S,v^W,v^F$ can be computed using (\ref{regular:eq}), as follows:
\begin{align}
v^S &= v^S \cdot P_{SS} + v^W \cdot P_{WS} + v^F \cdot P_{FS}, \nonumber \\ 
v^W &= v^W \cdot P_{SW} + v^W \cdot \epsilon, \nonumber \\
v^F &= v^S \cdot P_{SF} + v^W \cdot P_{WF} + v^F \cdot P_{FF}, \nonumber \\
v^S &+ v^W + v^F = 1. 
\end{align}
The solution of this set of equations gives the values:

\makeatletter
    \def\tagform@#1{\maketag@@@{\normalsize(#1)\@@italiccorr}}
\makeatother

\footnotesize
\begin{align}\label{eq:Values}
v^S &= \frac{(1-\epsilon) \cdot P_{FS}}{P_{FS} \cdot (1-\epsilon+P_{SW}) + (1-\epsilon) \cdot (1-P_{SS})- P_{WS} \cdot P_{SW}}.\nonumber \\
v^W &= \frac{P_{FS} \cdot P_{SW}}{P_{FS} \cdot (1-\epsilon+P_{SW}) + (1-\epsilon) \cdot (1-P_{SS})- P_{WS} \cdot P_{SW}}.\nonumber \\
v^F &= 1 - \frac{P_{FS} \cdot (1-\epsilon+P_{SW})}{P_{FS} \cdot (1-\epsilon+P_{SW}) + (1-\epsilon) \cdot (1-P_{SS})- P_{WS} \cdot P_{SW}}.\nonumber\\
&
\end{align}
\normalsize
Substituting $P_{WS}$ in the denominator by $1-\epsilon-P_{WF}$, the denominator can be written as $P_{FS} \cdot (1-\epsilon+P_{SW})+P_{SW}\cdot P_{WF}+(1-\epsilon) \cdot P_{SF}$, with all the terms being positive. The numerators can then determine the sign of the values of $v^S,v^W,v^F$. 
It is obvious that $v^S,v^W,v^F$  will have positive values if $P_{SW},P_{FS}$ are nonzero. However, if $P_{SW}$ equals zero, there will be no transition to the warning state when the CI starts at the success state. This implies that the warning state will be isolated which contradicts the fact that the chain is irreducible. The assumption $P_{FS}$ is zero also cannot hold from a practical point of view since in this case the infrastructure cannot be recovered to the success state hence is not resilient.

This shows that the values $v^S,v^W,v^F$ will always be positive hence they represent valid transition probabilities which proves that $\boldsymbol{P}$ is a regular Markov matrix.
\end{proof}

After deriving the convergence values, the resilience of an infrastructure will then be:
\begin{equation} \label{resindex:eq}
\gamma = \lim_{n \to \infty} \frac{1}{u^F_n}=\frac{1}{v^F}.
\end{equation}

To shed some light on the number of time steps, i.e., powers (iterations) needed for the matrix $\boldsymbol{P}$ to converge to $\boldsymbol{V}$, three different CIs are examined as shown in Fig.~\ref{fig:4} and Fig.~\ref{fig:5}.
Fig.~\ref{fig:4} shows the values of $P_S$ when it reaches a constant value after some time steps, similarly, Fig.~\ref{fig:5} shows the values of $P_W$.
In this example, the first CI has high probability of being at a success state, $P_{SS}=0.7$ and high probabilities of returning to the success state, $P_{WS}=P_{FS}=0.7$. Fig.~\ref{fig:4} and Fig.~\ref{fig:5} show that the convergence occurs at $n=3$ time steps. The second CI has a high probability of being at a success state $P_{SS}=0.8$ but lower transition probabilities $P_{WS}=0.2$ and $P_{FS}=0.3$. The convergence in this case occurs at $n=7$ time steps. Finally, the third CI has a low probability $P_{SS}=0.4$ and high transition probabilities $P_{WS}=0.9$ and $P_{FS}=0.7$ and, hence, its convergence occurs approximately at $n=6$ time steps.
Clearly, only a few time steps are needed for each CI's transition matrix to converge which corroborates the practicality of our proposed approach.
\begin{figure}[t]
  \centerline{\includegraphics[width=7cm]{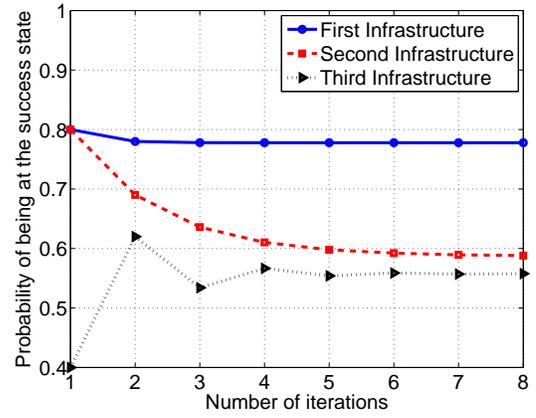}}
  \caption{$P_{S}$ convergence with number of iterations.}\label{fig:4}
\end{figure}

\begin{figure}[t]
  \centerline{\includegraphics[width=7cm]{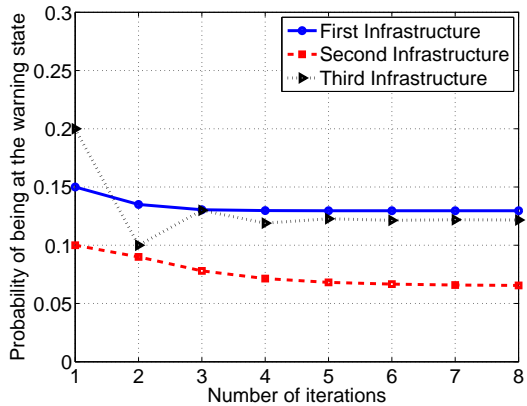}}
  \caption{$P_{W}$ convergence with number of iterations.}\label{fig:5}
   \vspace{-0.4cm}
\end{figure}

In our model, we are interested in studying the effect of improving the probability $P_{WS}$, on the CI's resilience.
This is because, we want to increase the transition probability from warning ($W$) state to the success state ($S$) thus reducing the probability of failure.
To this end, we evaluate the rate of change of the resilience index with respect to the probability $P_{WS}$, as given by:

\makeatletter
    \def\tagform@#1{\maketag@@@{\normalsize(#1)\@@italiccorr}}
\makeatother

\footnotesize
\begin{align}\label{eq:changerate}
&\hspace{0.1cm}\frac{\partial v^F}{\partial P_{SW}} = \nonumber\\ &\frac{- P_{FS} \cdot P_{SW} \cdot  (1-\epsilon+P_{SW})}{\big(P_{FS} \cdot (1-\epsilon+P_{SW}) + (1-\epsilon) \cdot (1-P_{SS})- P_{WS} \cdot P_{SW}\big)^2}.
\end{align}
\normalsize

This rate of change is strictly negative which implies that $v^F$ will always decrease with the increase of $P_{WS}$. From (\ref{eq:Values}) and (\ref{resindex:eq}), it can be clearly seen that the resilience will have a positive rate of change with respect to the probability $P_{WS}$.

Finally, we define the \emph{resilience index} $\theta$ of a CI as:

\vspace{-0.2cm}
\begin{equation}\label{eq:resIndexFinal}
\theta = \frac{\gamma}{\gamma_{\textrm{max}}}=\frac{v^F_{\textrm{min}}}{v^F},
\end{equation}
where $v^F_{\textrm{min}}$ is the minimum value of $v^F$ that can be achieved at the maximum value of $P_{WS}=1-\epsilon$ when substituted into (\ref{eq:Values}) and (\ref{resindex:eq}).
This value achieves a maximum resilience $\gamma_{\textrm{max}}$.
It is straightforward to show that $\theta$ is positive with $\theta \leq 1$.
The resilience index in this way helps to evaluate how far each CI is from its maximum achievable resilience. It can also help to compare different CIs as their resilience is measured on the same scale.
Fig.~\ref{fig:6} shows the values of the resilience index with the increasing values of $P_{WS}$ for different $\epsilon$ values when $P_{SS}=0.8,P_{FS}=0.5$.

\begin{figure}[t]
  \centerline{\includegraphics[width=7cm]{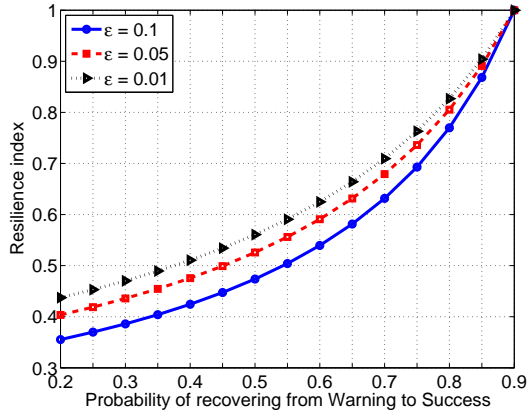}}
  \caption{Resilience index change with the probability $P_{WS}$.}\label{fig:6}
 \vspace{-0.4cm}
\end{figure}

Next, we study how to compute the probability $P_{WS}$ for an infrastructure and the effect of improving this probability on the resilience of a given CI. A Bayesian network is defined for this purpose as explained next. 

\section{Bayesian Network Model for CI Probability of Failure }\label{Sec:Bayesian}
	
To compute $P_{WS}$, we need to evaluate the probability of failure of a CI, given the probability of failure of each of its individual components. Since the failure of one or more components can cause other components to fail, we need to consider the relationship between the components when computing $P_{WS}$. To this end, a Bayesian network ~\cite{ICIP17} is a suitable framework. 

\subsection{Bayesian Networks: Preliminaries}

A Bayesian network is a network that describes the causality and relationship between independent random variables under incomplete information ~\cite{ICIP17}. A Bayesian network is normally represented by a directed acyclic graph (DAG) in which each node represents one random variable. Let $G(\mathcal{X},\mathcal{E})$ be a Bayesian network, then $\mathcal{X}=\{X_1,\dots,X_n\}$ is the set of nodes which represent different random variables and $\mathcal{E}$ is a set of directed edges. A directed edge from a node $X_j$ to $X_i$ means that node $X_i$ depends on the node $X_j$ and in this case $X_j$ is called the parent of node $X_i$. A node can have multiple parents and the set of parent for a node $X_i$ is given by $\pi(X_i)$. Every variable $X_i$ can take a value from a finite set of values, e.g., if the variables are binary they can either be true or false. Finally each node is associated with a conditional probability table (CPT) while roots, i.e, nodes without parents, are assigned direct probabilities. A CPT for a node $X_i$ gives the conditional probabilities between $X_i$ and every node in $\pi(X_i)$. 

Consider a Bayesian network with binary values, each variable in the network can be either true ($\pazocal{T}$) or false ($\pazocal{F}$) with a given probability. For a variable $C_1$, its $\pazocal{T}$ and $\pazocal{F}$ probabilities, i.e, $P(C_1)=\pazocal{T},P(C_1)=\pazocal{F}$ are written as $P(c_1),P(\bar{c}_1)$ where the lowercase letter indicates a value of the variable. The probabilities within each variable sum to $1$, i.e., $P(c_1)+P(\bar{c}_1)=1$. If a variable, e.g. $D_1$, has two parents $C_1$ and $C_2$, then the CPT of $D_1$ will have eight entries representing the possible combinations of $C_1$ and $C_2$ with the $\pazocal{T}$ and $\pazocal{F}$ values of $D_1$. However, as the $\pazocal{T}$ and $\pazocal{F}$ values for any variable sum to $1$, only half of the CPT entries must be stored, i.e., the $\pazocal{T}$ values of $D_1$.

Once the probabilities and the CPTs are assigned, probabilistic inference can be performed to calculate the probability of any variable given some evidence in the network. Calculating probabilistic inference in general Bayesian networks in known to be NP-Hard ~\cite{ICIP21}, however, Pearl ~\cite{ICIP18} introduced a polynomial time algorithm to perform probabilistic inference in \emph{singly connected Bayesian networks}. A singly connected Bayesian network, also known as a polytree, is a Bayesian network where it has no loops, i.e., there is only one path between any two nodes in its underlying undirected graph.

\subsection{Evaluating CI Probability of Failure}

For our resilience problem, we introduce a Bayesian network to model the possible failure events of a given CI that can prevent it from delivering its designated service.
We model the various possible failure events of the components of a CI which can lead to total CI failure with a given probability.
The total failure probability calculated from this Bayesian network will effectively represent the transition probability $P_{WF}$ as this is the probability with which a partial failure causes a total failure.
On the other hand, failure events, such as natural disasters, that can cause a sudden CI failure, are not captured by this network. They can be modeled in a separate network to calculate the transition probability $P_{SF}$.

The Bayesian network is constructed such that failure events are modeled as variables (nodes) in a hierarchical way. Nodes are grouped into levels where failures in one level can cause failures in the next level. Fig.~\ref{fig:3} shows the structure of the proposed Bayesian network in which the number of nodes and levels vary according to each infrastructure.
Roots in the network, $(C_1,C_2,C_3,D_2,\dots,D_n)$, represent possible failure events to respective CI components. Failure, here, can happen due to external effects or normal wear-and-tear of the components.
The subsequent levels represent the cascading failure to other major components, e.g. $Z_1$ and $Z_2$.
These major failures can, in turn, cause a failure in their next level, with given probabilities, and so on until the whole infrastructure fails. The CI failure is represented by the single leaf in the Bayesian network.

\begin{figure}[t]
  \centerline{\includegraphics[width=8.8cm]{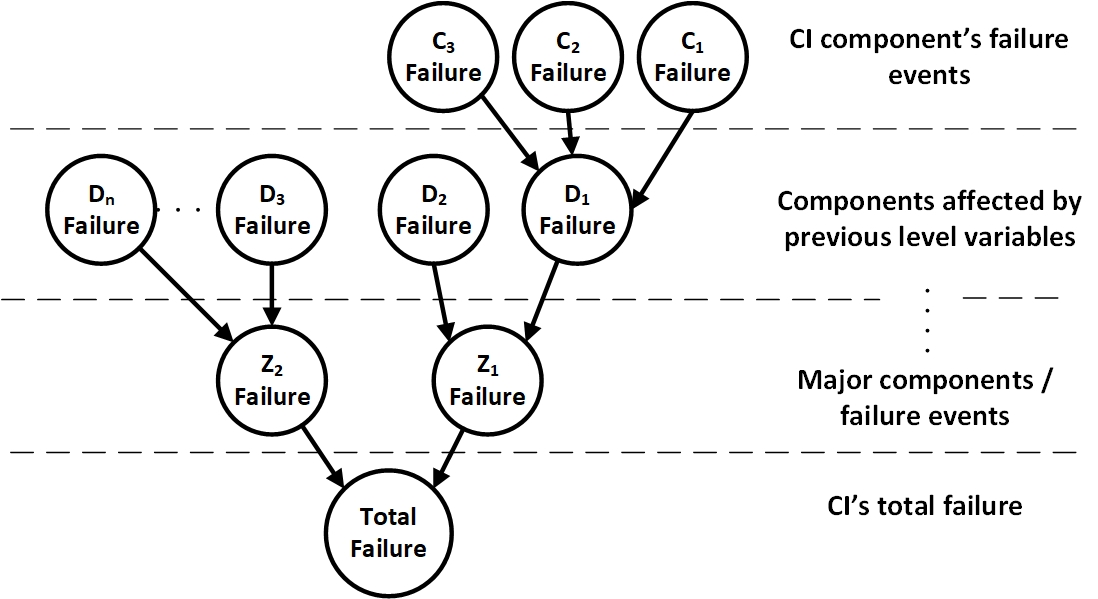}}
  \caption{A Bayesian network representing the hierarchical failure events within a CI's components.}\label{fig:3}
\end{figure}
As the variables represent failure events, they can be either true ($\pazocal{T}$) or false ($\pazocal{F}$) with a given probability. A $\pazocal{T}$ value implies that the failure event has occurred and $\pazocal{F}$ means it has not occurred. Root variables are assigned $\pazocal{T}$ and $\pazocal{F}$ probabilities, while the variables at subsequent levels are associated with conditional probabilities for the possible combinations of their parents' values. Probabilistic inference can then be performed to compute the total probability of failure, of a CI according to the failure probabilities of its components. Note, the proposed network in Fig.~\ref{fig:3} is a singly connected Bayesian network and, hence, probabilistic inference can be performed in polynomial time. 

To compute the probability of failure $P_{WF}$ for any given CI, we start by calculating the \emph{prior marginal probability} of failure. This probability is calculated from the initial assigned probabilities.
Assume, without loss of generality, that $X_n \in \mathcal{X}$ is the variable representing the failure, then the prior marginal probability of $x_n$, the $\pazocal{T}$ value of $X_n$, can be calculated in a manner analogous to ~\cite{ICIP22}:
\begin{equation}\label{eq:prior}
P(x_n) = \sum_{\mathcal{X} \setminus X_n} P(x_1,\dots,x_n),
\end{equation}
where $P(x_1,\dots,x_n)$ is the joint probability for all the instantiations of the independent random variables $X_1,\dots,X_n$. The summation is calculated over all variables except $X_n$, thus these variables are marginalized from the joint probability. The joint probability in (\ref{eq:prior}) is given by:
\begin{equation}\label{eq:joint}
P(x_1,\dots,x_n)=\prod_{i=1}^{n}P(x_i|\pi(x_i)),
\end{equation}
where $P(x_i|\pi(x_i))=P(x_i)$ when the set $\pi(X_i)$ is empty.

This prior marginal probability $P(x_n)$ represents the initial $P_{WF}$ in our CI Markov chain model. It can be used to calculate the initial resilience index of a given CI.
Then, the effect of each node on improving the resilience index can be calculated and, hence, the components of a CI can be prioritized based on their effect on $P_{WF}$. However, the effect of each component should not be considered separately, as securing a component will reduce $P_{WF}$ which will also reduce the effect of other components on $P_{WF}$. The proposed procedure for sorting the components is given next.

We calculate the \emph{posterior marginal probability} of failure given the evidence of each root variable separately. Let the number of root variables be $m<n$. Then, the marginal probability will be given by ~\cite{ICIP22}:
\begin{equation}\label{eq:firstmarginal}
P(x_n|\bar{x_i}) = \sum_{\mathcal{X} \setminus X_n} P(x_1,\dots,x_n|\bar{x_i}), i \in \{1,\dots,m\},
\end{equation}
where $\bar{x_i}$ is the evidence value for the variable $X_i$ when $P(X_i)=\pazocal{F}$.
We consider the false (failure) probability to capture the positive effect of a variable on the total probability of failure.
This is calculated for all root variables and the values are sorted in a descending order.
The variable that causes the greatest reduction in $P_{WF}$ then represents the first component of the CI that must be overhauled.
This variable is also used as a new evidence variable in the Bayesian network to determine the second most affecting variable (component) from the remaining roots.

Assume without loss of generality that $X_1$ is the root with the most effect on $P_{WF}$, then the next posterior marginal probability is calculated considering only the $\bar{x_1}$ instantiation of $X_1$. The probability is calculated for the remaining root variables individually as given by:
\begin{align}\label{eq:secondmarginal}
P(x_n|&\bar{x_1},\bar{x_i}) = \nonumber\\
& \sum_{\mathcal{X} \setminus X_n,X_1} P(x_1,\dots,x_n|\bar{x_1},\bar{x_i}),
 i \in \{2,\dots,m\}.
\end{align}

\begin{algorithm}[t]
\DontPrintSemicolon
  \caption{Variables Sorting According to Their Effect}
 \label{Alg1}
  \KwIn{Bayeian network variables $\mathcal{X}$ where $X_n \in \mathcal{X}$ is the only leaf representing the failure}
   \KwOut{A sorted set of the root variables $\mathcal{S}$}
    \Begin{
	Calculate the prior marginal probability of failure $P(x_n)$\;
	\For{each root node}{
	Calculate the posterior marginal probability $P(x_n|\bar{x_i})$\;
	Calculate the variable's effect $P(x_n)-P(x_n|\bar{x_i})$
	}
	Sort variables descendant according to their effect\;
	Determine the variable $X_i$ with the greatest effect\;
	Store the variable $X_i$ in the set $\mathcal{S}$\;
	Define the partial set $\mathcal{X}_r$ as the set of roots excluding $X_i$\;
	\While{The set $\mathcal{X}_r$ is not empty}{
	\For{each node $X_j \in \mathcal{X}_r$}{
	Calculate the posterior marginal probability $P(x_n|\bar{x_i},\bar{x_j})$\;
	Calculate the variable's effect $P(x_n|\bar{x_i})-P(x_n|\bar{x_i},\bar{x_j})$\;
	}
	Determine the variable $X_j$ with the greatest effect\;
	Store the variable $X_j$ in the set $\mathcal{S}$\;
	Update the set $\mathcal{X}_r = \mathcal{X}_r \setminus X_j$
	}
   \Return{$\mathcal{S}$}
      }
\end{algorithm}
\setlength{\textfloatsep}{0pt}

This procedure is applied to all the roots adding one root to the evidence variables each time. The procedure will end by sorting all the components of a CI in a descending order according to their effect on the probability of failure. The steps of this procedure are summarized in Algorithm ~\ref{Alg1}.

Note that, according to (\ref{eq:joint}), the joint probability considers the parents of each node. Thus, (\ref{eq:secondmarginal}) can be derived from (\ref{eq:firstmarginal}) by considering changes in the branch between the leaf node and the new variable only. All the other summations in (\ref{eq:firstmarginal}) will not change as the evidence variable does not belong to this branch. This allows a reduction in the complexity of calculating the updated probabilities after fixing the components. 

Algorithm ~\ref{Alg1} can then be used by any CI to determine the order according to which it must fix its components. As CIs typically allocate resources to improve their resilience~\cite{ICIP34,ICIP33,ICIP29,ICIP32,ICIP24}, Algorithm ~\ref{Alg1} can help CIs to determine the components to which resources will be allocated within each CI. 

Here, we note that, in practice, CIs operate within larger systems (e.g., an entire city) that are composed of multiple, interdependent CIs that collectively provide a common service. As such, the function loss of one CI will impact other interdependent CIs and, therefore, when analyzing the resilience of a large-scale system, one must consider all the interdependent CIs. This, in turn, brings forward a new problem of allocating resources, such as monitoring devices among a system of multiple CIs which is addressed next. Within the context of resource allocation, Algorithm ~\ref{Alg1} is applied by each CI to make the best use of its allocated resources.

\setlength{\textfloatsep}{17pt}

\section{Resource Allocation for Optimized Resilience}\label{Sec:Resources}

As evident from the previous discussion, our next step is to study the problem of allocating resources in a system of multiple CIs,  while taking into account the individual Bayesian network model of each CI. Resources can range from cyber resources to personnel or physical equipment. We classify resources into two categories: preventive and rapid intervention resources. Preventive resources are resources that help CI's  components become less vulnerable to failures. This might include replacing some components with more reliable ones or installing redundant components. Rapid intervention resources, on the other hand, requires monitoring and alarming systems to be deployed and requires the existence of on-site facilities that can be used to fix or replace corrupted components in a timely-manner. The choice of either category of resources depends on the nature of the infrastructure and the cost of using each. For instance, in a power plant, preventive resources can represent installing redundant switches or replacing old stators, while rapid intervention resources can represent excessive monitoring of the generators to repair any defects once they occur to help keep the generator working.

As resources are infrastructure-specific, we introduce an application-specific case study to highlight the importance of our framework. Though the framework can be applied to any CI, studying the problem of resource allocation within the context of a specific CI, as a case study, helps better illustrate our framework, as shown next.
\vspace{-0.1cm}
\subsection{Hydropower dams: A case study}\label{Sec:Dam}

We apply the proposed framework to hydropower dams and their impact on power systems as a practical CI in order to measure the resilience improvement that can be achieved. Dams are classified as one of the critical infrastructure sectors according to the US DHS~\cite{CIP01}. Hydropower dams provide a good platform to apply our proposed framework as they have many connected components that could be affected by numerous failure events.
Recall that, according to our proposed framework, failure will be defined as the inability of the dam to produce electricity.

A Bayesian network is designed for each dam where the parents to the node representing failure are the dam's main components such as penstocks, generators, turbines and transformers. In turn, these variables are modeled as children nodes of the variables representing smaller components such as stators, rotors, intake gates, and blades. Components are connected in a hierarchical manner until the roots that represent small components failure. Fig.~\ref{fig:17} shows a scheme for a hydropower dam highlighting its main physical components along with part of the Bayesian network defined for this dam. We use previous failure statistics and reliability analyses ~\cite{ICIP23} to assign probabilities of failure to the roots of the Bayesian network. Conditional probabilities between components are assigned based on the components' relations similar to method used in fault trees. Note that the same method of assigning probabilities can be applied to any other CI.

\begin{figure}[t]
  \centerline{\includegraphics[width=7.4cm]{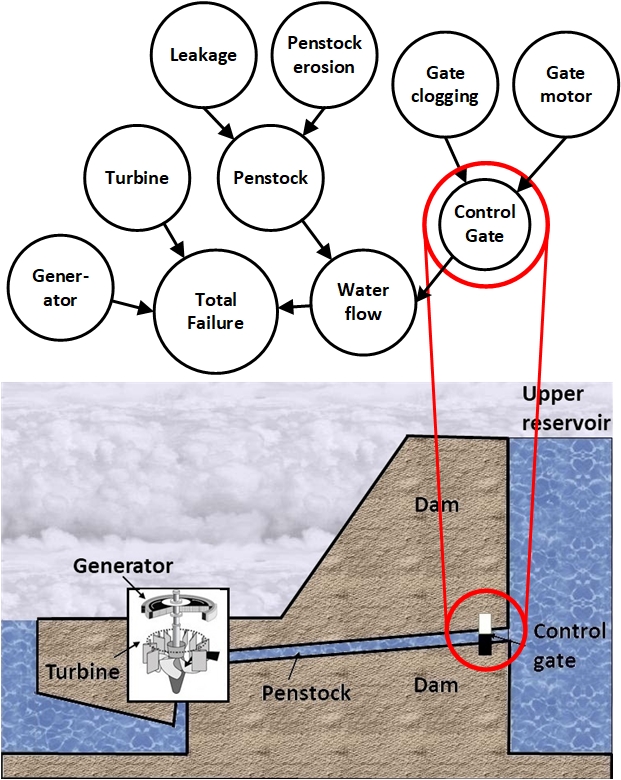}}
  \caption{A Bayesian network model for the hydropower dam in which each node defines the failure probability for one of the physical components all the way to the total CI failure.}\label{fig:17}
   \vspace{-0.3cm}
\end{figure}

From a resources perspective, preventive resources are seen as a long-term solution to improve the resilience of dams in service. Preventive resources require some components to be replaced, which might not be applicable when the dam is in service. Therefore, \emph{we focus on rapid intervention resources} which include monitoring devices, such as sensors and cameras, and maintenance equipment. We propose to use both fixed sensors and drones in the monitoring process. Drones can be used in general to inspect areas of interest in CIs~\cite{ICIP08} and help to inspect hard-to-reach points where conventional sensors/monitoring methods cannot be used. Recently, the use of drones to inspect even the inner parts of the dam was shown to be applicable in~\cite{ICIP09}. In this work, the authors modified a drone and used it to inspect the inside of the pentstocks of a number of dams. Mechanical robots on the other hand can be used to inspect a dam's key sections that cannot be reached by drones such as underwater components~\cite{ICIP11}. 

The majority of dams in the United States, are privately owned ~\cite{ICIP12} and their owners are responsible for their safety. However, there is still a federal role for ensuring dams' safety as dams can severely affect persons and properties in case of failure.
The same applies to electricity supply, the failure of a dam to generate electricity will affect huge parts of the electric grid that it supplies.
These facts reveal the importance of having a system operator to manage the process of resource allocation within multiple dams. The system operator is considered as a centralized agency that provides the resources to a dam, or more, to increase their resilience and hence can avoid long interruptions to the electric service.
Having a system operator that can manage the resources, especially drones, is useful as the operation of drones is regulated by the federal aviation administration (FAA)~\cite{ICIP10} and is not granted to all private organizations. In the following, we will use a general notion of resources, without being restricted to drones, as the framework can be applied to any type of resources.

Next, we formulate the problem of allocating resources within a system of multiple dams (CIs). We propose to use \emph{contract theory}, a powerful framework from microeconomics that provides useful tools for designing contractual agreements between a principal and a number of agents~\cite{CT00}. The system operator is modeled as the principal and the owner of dams as agents, as discussed next in more details.

\subsection{Resource Allocation using Contract Theory}\label{Sec:contract}

We consider a system in which an electric grid operator, referred to as the principal, is interested in providing a number of resources to the owners of dams to be used in the process of surveillance and rapid intervention. Let $\mathcal{N}$ be the set of $N$ targeted dams. Dams are assumed to have different owners. These dams, being part of the grid, sell their generated electricity to the power grid managed by the principal. Each dam can utilize the resources to improve its resilience and hence reduce the probability of failure to generate electricity. 

The principal has a limited number of discrete (integer valued) resources $R$ to be allocated to the dams and, hence, it decides on how to optimally use these resources. The goal of the principal is to invest in improving the probability of power generation and, hence, decreasing the probability of losses due to a dam's failure.
We model the principal's payoff as the difference between the total rewards it gets from the dams' owners and the total expected losses due to each dam's failure.
Losses are modeled as a function of the total probability of failure before and after the deployment of resources.
The total utility $Z_{p}(\boldsymbol{R})$ that the principal achieves as a function of the vector of resource allocation $\boldsymbol{R}$ is given by:
\vspace{-0.1cm}
\begin{equation}\label{eq:PrincipalUtility}
Z_{p}(\boldsymbol{R}) = c \cdot \sum_{i=1}^{N} R_i - \sum_{i=1}^{N}(\alpha_{f_i} - \alpha_{R})\cdot B_i(R_i) \cdot n_i \cdot P_i,
\vspace{-0.1cm}
\end{equation}
where $\boldsymbol{R}$ is the resource allocation vector across all dams with each element $R_i$ specifying the number of resources allocated to dam $i$, $\alpha_{f_i}$ is the expected real-time energy price if dam $i$ fails to generate electricity, $\alpha_{R}$ is the average real-time energy price in normal operation, $P_i$ is the contracted power production for dam $i$, $n_i$ the expected number of hours the dam will be out-of-service due to failure, and $c$ is the monetary reward the principal gets for a unit of resources which can also be seen as a cost. Note that, the resources in (\ref{eq:PrincipalUtility}) refers to monitoring resources or drones as discussed earlier.
We introduce the function $B_i(R_i)$ to measure the improvement in the resilience of a CI $i$ due to the amount of allocated resources. Specifically, the CI evaluates the difference in its resilience index before and after using the resources $R_i$, and $B_i(R_i)$ is given as:
\begin{equation}\label{eq:B-function}
B_i(R_i)= \gamma_{i_{R_i}}^{-1} - \gamma_{i}^{-1} = v^F_{R_i}-v^F,
\vspace{-.05cm}
\end{equation}
where $\gamma_{i_{R_i}}$ and $v^F_{R_i}$ are the values calculated from (\ref{eq:Values}) and (\ref{resindex:eq}) respectively for the updated values of $P_{WS}$.
These updated values are calculated from the Bayesian network as the result of fixing a number of variables equal to $R_i$ according to the order specified by Algorithm \ref{Alg1}. The effect of the first unit of resources on $P_{WS}$, for each dam, is calculated from (\ref{eq:firstmarginal}) while the effect of the remaining resources is calculated from (\ref{eq:secondmarginal}) for each additional unit of resources. Without loss of generality, we assume that each component of a CI can be secured by a single unit of resources.

Each dam's owner will evaluate the amount of resources it receives based on the resilience enhancement that will result from the allocated resources. This resilience improvement is reflected by a higher probability of generating power and, hence, a higher probability to sell the generated power with the real-time prices. 
The utility $Z_{d_i}(R_i)$ of dam $i$ is defined as follows:
\begin{equation}\label{eq:DamUtilitiy}
Z_{d_i}(R_i)= \alpha_{d_i} \cdot B_i(R_i) \cdot n_i \cdot P_i - c \cdot R_i,
\end{equation}
where $\alpha_{d_i}$ is the average day-ahead energy price for dam $i$. The remaining parameters are similar to those in (\ref{eq:PrincipalUtility}).

The principal wants to offer contracts to the dam's owners to maximize its utility in (\ref{eq:PrincipalUtility}). A \emph{contract}~\cite{CT00} can be seen as an agreement between the principal and the dam's owner using which the principal provides and operates resources to monitor, inspect, and fix points of interest in the dam and gets monetary rewards in return. Every contract is defined as a pair $(R_i,c \cdot R_i)$ representing the amount of resources and the monetary reward (cost) the dam's owner should pay for these resources. 

In our model, we assume the principal has complete information about the targeted dams. This information should be provided by each dam's owner as the resource evaluation, from the Bayesian network, is dam-specific and cannot be estimated by the principal without the dam owners. Moreover, the principal, being the system operator, already knows all of the other parameters. Hence, the focus of the principal is to design contracts in a way to ensures each dam's owner participation in order to maximize its total benefit. Contracts offered by the principal should then satisfy the key property of \emph{individual rationality}, under which each dam's owner is interested to participate only if the benefit it gets is greater than or equal to the amount it pays , i.e.,
\begin{equation}
\alpha_{d_i} \cdot B_i(R_i) \cdot n_i \cdot P_i - c \cdot R_i \ge 0.
\end{equation}

The principal can then design the optimal contracts by maximizing  its utility and satisfying the constraints as follows:
\begin{equation}\label{eq:opt}
\max_{R_i}  \hspace{0.2cm} c \cdot \sum_{i=1}^{N} R_i - \sum_{i=1}^{N}(\alpha_{f_i} - \alpha_{R})\cdot B_i(R_i) \cdot n_i \cdot P_i,
\end{equation}
\vspace{-0.5cm}
\begin{align*}
&\textrm{ s.t. } \hspace{0.5cm} 
\alpha_{d_i} \cdot B_i(R_i) \cdot n_i \cdot P_i - c \cdot R_i \ge 0, \ i \in \mathcal{N}, \\  \nonumber
&\hspace{1.1cm} \sum_{i=1}^{N}R_i =  R.\\ \nonumber
\end{align*}
\vspace{-1cm}

\subsection{Optimal Contract}\label{Sec:solution}

Solving the problem in (\ref{eq:opt}) is challenging as the function $B_i(R_i)$ is not continuous. $B_i(R_i)$ has discrete values for a finite set of $R_i$ values. To address this challenge, we first start by inspecting the properties of the function $B_i(R_i)$ in order to solve the problem in (\ref{eq:opt}).

\begin{proposition}\label{lem1}
The values of the function $B_i(R_i)$ represent a monotonically increasing concave sequence.
\end{proposition}

\begin{proof} We prove this proposition by showing how the values of $B_i(R_i)$ are calculated. Let the values $a_{i-1},a_i,a_{i+1}$ be any three consecutive values for the improvement in $P_{WF}$ achieved by fixing any three consecutive variables as calculated from Algorithm \ref{Alg1}. These values satisfy the following two properties:
\begin{align}
a_{i-1} \ge a_i \ge a_{i+1},\nonumber\\
a_{i-1} - a_i \ge a_i - a_{i+1},
\end{align}
according to the selection criteria defined in Algorithm \ref{Alg1} in which the biggest improvement is captured first.

Let $b_{i-1},b_i,b_{i+1}$ be the values calculated from (\ref{eq:B-function}) for the updated $P_{WS}$ values for $a_{i-1},a_i,a_{i+1}$, respectively. Each $b_i$ is the difference between the updated $v^F_{R_i}$ and the current $v^F$.
According to (\ref{eq:changerate}), the values of $v^F_{R_i}$ are inversely proportional to the $P_{WS}$ values, hence, the values $b_{i-1},b_i,b_{i+1}$ follow the same relation and it can be seen that $b_i=f(\frac{1}{a_i})$. Therefore, we can conclude that the sequence is monotonically increasing:
\begin{align}
b_{i-1} \le b_i \le b_{i+1},
\end{align}
and the difference relation becomes:
\begin{align}
b_{i+1} - b_i \le b_i - b_{i-1}.
\end{align}
Rewriting the last inequality we get:
\begin{align}
b_{i+1} + b_{i-1} \le 2 \cdot b_i,
\end{align}
which proves that the sequence is concave ~\cite{ICIP25}. 
\end{proof}

Using Proposition \ref{lem1} with the first constraint in problem (\ref{eq:opt}), we can see that the constraint is a difference between a monotonically increasing concave sequence $\alpha_{d_i} \cdot B_i(R_i) \cdot n_i \cdot P_i$ and a strictly increasing linear sequence $c \cdot R_i$.
The result will be a concave sequence that can have both positive and negative values depending on the difference between the two sequences.
This result is used by the principal to determine the range of values, i.e. $\left[R_{i_{\textrm{min}}},R_{i_{\textrm{max}}}\right]$, that meets the first constraint and, hence, will be acceptable for the dam's owners as it satisfied individual rationality.

Next, we study the properties of the objective function in (\ref{eq:opt}) based on the results of the previous proposition.

\begin{lemma}\label{lem2}
The objective function in (\ref{eq:opt}) is a convex sequence with respect to each dam that is monotonically increasing.
\end{lemma}

\begin{proof} We prove this lemma by showing the relation between the terms of the objective function. For a given dam $i$, the objective function is the difference between the reward $c \cdot R_i$ and a constant number $(\alpha_{f_i} - \alpha_{R}) \cdot n_i \cdot P_i$ multiplied by the concave function $B_i(R_i)$. The reward term represents a linear strictly increasing function in the number of resources $R_i$, while the second term is monotonically increasing concave function. Clearly, the difference between both terms will be a monotonically increasing convex sequence.
\end{proof}

\vspace{-0.5cm}According to Lemma \ref{lem2}, while being convex, the objective function might have negative values until a certain amount of resources is used, that is when the second term is higher than the rewards term. The principal can use this value to update the minimum number of resources, i.e. $R_{i_{\textrm{min}}}$, that should be allocated to each dam to represent a feasible solution to the principal. The update is done based on the larger of the two minimum values calculated in proposition \ref{lem1} and Lemma \ref{lem2}.
  
After determining the range of possible values for each allocation, the principal can use dynamic programming optimization~\cite{ICIP35} techniques to calculate the solution to the problem in (\ref{eq:opt}). In the dynamic programming representation, stages will represent the current allocation of resources for each dam. The state of each stage represents the current value of the objective function. The update from a stage to another, i.e., from an allocation to another, aims to increase the value of the objective function. If no increase can be achieved at one stage, then the current allocation is the optimal. This is shown next.

\begin{theorem}
The optimal resource allocation can be found using dynamic programming by updating the number of resources assigned to each dam at each stage while maximizing the benefits of each allocation. 
\end{theorem}

\begin{proof}
The values of the objective function are calculated at each value of the resources $R_i$ in the feasible range for each dam $i$, $\left[R_{i_{\textrm{min}}} , R_{i_{\textrm{max}}}\right]$. These values are stored for all dams as a matrix of size $N\cdot R$. Each value $z_{i,k}$ in the matrix is the objective function value evaluated for dam $i$ when it is assigned a number of resources $k$. The values of each row represent a monotonically increasing sequence as shown in Lemma \ref{lem2}.

The first stage in the problem starts by allocating the maximum feasible resources $k$ to the dam $i$ with the highest objective function value, i.e., $R_i = k$ for the dam with $R_{i_{\textrm{max}}}=k$ and $z_{i,k}$ is the largest among all other dams.
As the values for this dam $i$ represent a convex sequence and are monotonically increasing, the principal will not gain more by assigning a lower number of resources to this specific dam.
Since this value of $z_{i,k}$ is the maximum among all dams, this allocation represents the maximum value that the principal can get for this number of resources $k$. The rest of the resources, i.e., $R-k$ are assigned to dam $j$ having the largest $z_{j,{R-k}}$ among all dams.

This allocation represents the optimal solution at the first stage as the principal gets the highest utility for the current resource configuration. The principal then tries to get a higher utility by changing the current allocation scheme. The principal might be able to do so by decreasing the number of resources $k$ assigned to dam $i$ and assigning the difference to another dam $j$ if the following condition holds:
\begin{equation}
z_{i,k}-z_{i,{k-1}} < z_{j,{R-k+1}}-z_{j,{R-k}}  \ ,j=1,\dots,N,  j\neq i,
\end{equation}
where $k$ decreases by one unit of resources each time. The principal compares the utility gains that it can achieve by assigning more resources to another dam. These resources are taken from the dam with most resources. The current allocation ensures that the principal gets the largest utility from the allocated resources, so it is the optimal solution at this stage. Here, if the principal cannot increase its utility by changing a unit of resources, then it will try to change more than one at a time until the condition is satisfied or all the values of resources are checked.

The procedure continues at each stage by assigning less resources to the dam with the most resources if a higher utility can be achieved by allocating these resources to another dam.
This ensures that the principal achieves its maximum utility at each stage.
This procedure by starting at the final allocation and moving backward ensuring the maximum utility is achieved at each stage satisfies the Bellman equation~\cite{ICIP36} which is the necessary optimality condition in dynamic programming, Hence, the procedure achieves the optimal resource allocation.
\end{proof}

The complexity of calculating the optimal resource allocation for the previous dynamic programming problem is $O\big(N \cdot (R_{\textrm{max}} - R_{\textrm{min}})\big)$ where:
\begin{align}
R_{\textrm{max}} = \textrm{max}(R_{i,{\textrm{max}}}), \  i=1,\dots,R,\nonumber\\
R_{\textrm{min}} = \text{min}(R_{i,{\textrm{min}}}), \  i=1,\dots,R,
\end{align}
as for each number of resources in the range $\left[R_{\textrm{max}} - R_{\textrm{min}}\right]$, the program at most compares the utility function $N$ times for each of the $N$ dams.

\begin{figure}[t]
  \centerline{\includegraphics[width=7cm]{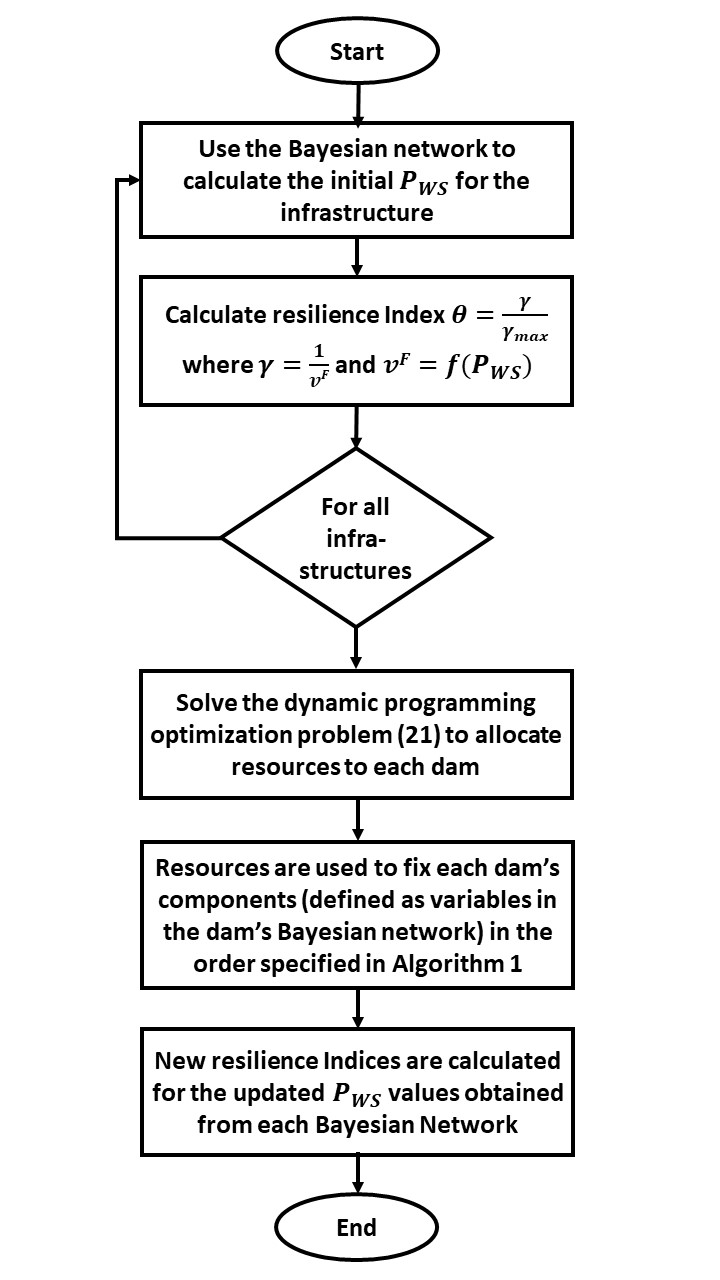}}
  \caption{Flowchart for the proposed mechanism.}\label{fig:7}
  \vspace{-0.3cm}
\end{figure}

Finally, we summarize our framework steps in the flow chart shown in Fig. ~\ref{fig:7}. Note that the problem discussed in this section, though discussed within the context of a case study, can be to used to allocate resources in other systems of multiple CIs.
Selecting a specific CI, hydropower dams here, helped to design meaningful Bayesian networks and to define CIs utilities in the contract-based allocation.
Next, we show some numerical results built on the selected case study, i.e., hydropower dams.

\section{Numerical Analysis and Results}\label{Sec:results}
Although our framework can be applied to any number of dams, for our simulations, we consider a case of two dams, in order to better highlight each dam's effect on the process of resource allocation. Two Bayesian networks are designed for the two dams using similar components but with different probabilities. In the following experiments, the transition probabilities are assumed to be the same for both dams $P_{SS}=0.8,P_{SW}=0.15,P_{FS}=0.5$, and $\epsilon=0.1$. However, the first dam is assumed to have a lower initial $P_{WF}=0.37$ while the second dam will have $P_{WF}=0.7$. Other parameters are set as follows: $\alpha_{d_1}=\$26,\alpha_{d_2}=\$20,P_1=120$~MW/h, $P_2 = 150$~MW/h, $n_1 = 30$~hours, $n_2= 20$~hours, $\alpha_R=\$33,\alpha_{f_1}=\$40$, and $\alpha_{f_2}=\$46$. $\alpha_{f_2}$ is assumed to be higher than $\alpha_{f_1}$ as $P_2$ is assumed to be higher than $P_1$, so the failure of the second dam will have a larger effect on increasing the prices of power. 

In Fig.~\ref{fig:8}, we show how the total probability of failure $v^F$ can be reduced by overhauling each dam's components. The components are overhauled using the allocated resources in the order specified by Algorithm 1 then the variables representing these components are adjusted in the Bayesian network. Note that, when $v^F$ decreases, both the resilience $\gamma$ and the resilience index $\theta$ increase according to (\ref{resindex:eq}) and (\ref{eq:resIndexFinal}). We can see that the second dam achieves a higher reduction in the probability of failure $v^F$. This because the second dam has a higher initial $P_{WF}$, so the difference between the initial and final $P_{WF}$ is higher resulting in a higher difference in $v^F$. This difference represents the function $B_i(R_i)$ as in (\ref{eq:B-function}). Fig.~\ref{fig:8} also corroborates Proposition ~\ref{lem1} by clearly showing that the improvement in each dam follows a monotonically increasing concave sequence. 

\begin{figure}[t]
  \centerline{\includegraphics[width=7cm]{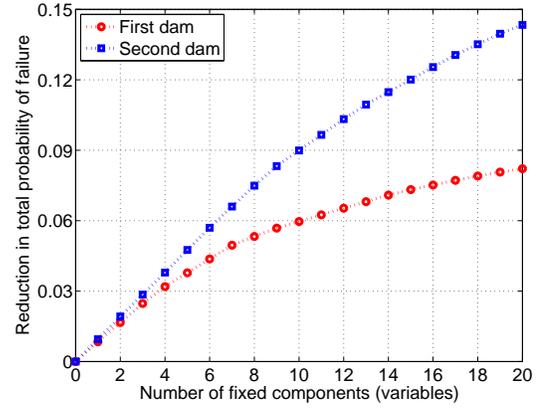}}
  \caption{Dam's probability of failure improvement with the number of overhauled components (their related variables are adjusted in the Bayesian network).}\label{fig:8}
  \vspace{-.2cm}
\end{figure}

Fig.~\ref{fig:9} shows the benefit that each dam receives
from the allocated resources, evaluated as the first term of (\ref{eq:DamUtilitiy}) before subtracting the cost, which is a function of $B_i(R_i)$. The figure shows the cost of the resources separately. The intersection between the cost and each dam's benefit will represent the range of values $\left[R_{i_{\textrm{min}}} , R_{i_{\textrm{max}}}\right]$ that each dam $i$ is willing to accept.
Any additional resource beyond $R_{i_{\textrm{max}}}$ will yield a negative dam's utility as the cost will be higher than the dam's benefit.
The range can be seen from Fig.~\ref{fig:9} to be $\left[0,13\right]$ and $\left[0,15\right]$ for the first and second dams, respectively. The first dam has a smaller range as its utility is lower than the second dam, starting at $12$ units of resources. This utility is lower as the first dam has a smaller power production $P_1$ and a higher initial resilience index that causes the changes in $B_1(R_1)$ to be small.
\begin{figure}[t]
  \centerline{\includegraphics[width=7cm]{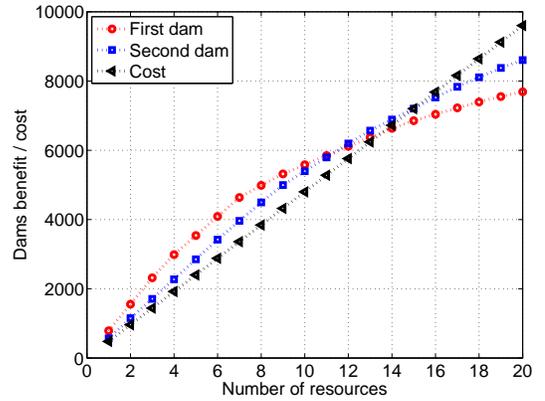}}
  \caption{Dam's benefit and resources cost. The intersection represents the range of resources each dam is willing to accept.}\label{fig:9}
    \vspace{-.29cm}
\end{figure}

In Fig.~\ref{fig:10}, we show the utilities of the dams as given by (\ref{eq:DamUtilitiy}). We can see that both dams have nonnegative utilities only in the ranges discussed before, i.e., $\left[0,13\right]$ for the first dam and $\left[0,15\right]$ for the second dam. Fig.~\ref{fig:10} also shows that both utilities are concave and each has a maximum value at a certain amount of resources. The first dam has its maximum utility when it uses $7$ units of resources, while the second dam can achieve its maximum at $9$ units of resources. These values represent the maximum distance between the benefit and the cost in Fig.~\ref{fig:9}. Note that, according to the proposed framework, the owners of the dams are willing to accept resources in their feasible ranges regardless of their maximum utility. This is because the extra resources will help improve their resilience.

\begin{figure}[t]
  \centerline{\includegraphics[width=7cm]{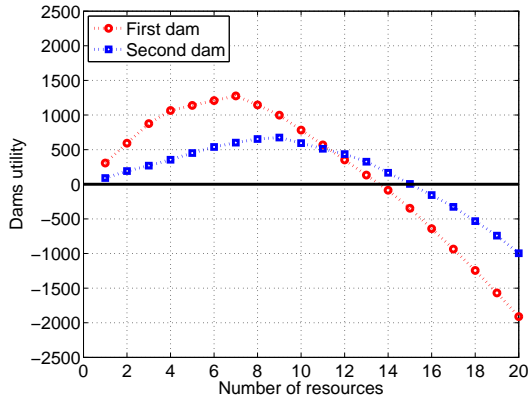}}
  \caption{Dam's utilities relation with the number of allocated resources.}\label{fig:10}
    \vspace{-.4cm}
\end{figure}

In Fig.~\ref{fig:11}, we show the principal's utility calculated for each dam. Fig.~\ref{fig:11} shows two curves: one for each dam where the second curve is plotted upside down. The horizontal axes show the amount of resources allocated to each dam, while the remaining resources are allocated to the other dam. Therefore, the principal's total utility, at each allocation, is the summation of the values from the two curves corresponding to this allocation. Fig.~\ref{fig:11} corroborates the result of Lemma ~\ref{lem2} where we showed that the principal's utility calculated for each dam separately represents a monotonically increasing convex sequence. From Fig.~\ref{fig:11}, we can see that the principal achieves a higher utility by allocating resources to the first dam. This stems from the fact that the first dam has a lower power production $P_1$ and a higher initial resilience, i.e., lower $B_1(R_1)$.
Fig.~\ref{fig:11} also has two solid vertical lines, each of which representing the maximum number of resources $R_{i_{\textrm{max}}}$ for a dam. Any allocation of resources beyond these lines will no longer be feasible as it yields negative dams' utilities. The lines correspond to the maximum resources in the feasible range for each dam ($13$ for the first dam and $15$ for the second dam). The optimal allocation in this case is to allocate $13$ units of resources to the first dam and $7$ units of resources to the second dam.

\begin{figure}[t]
  \centerline{\includegraphics[width=7.2cm]{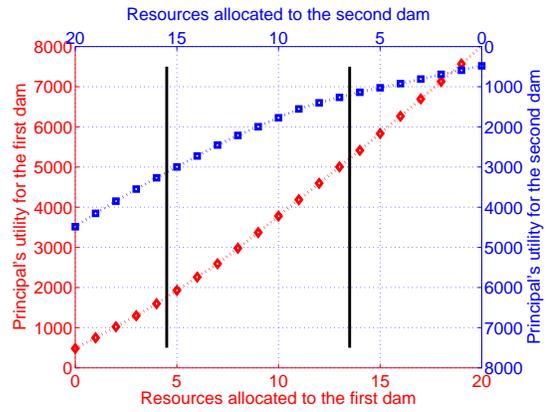}}
    \vspace*{.1cm}
  \caption{The principal's utility for each individual dam with respect to the number of the allocated resources to this dam.}\label{fig:11}
  \vspace{-0.2cm}
\end{figure}

%

In Fig.~\ref{fig:13}, we show the resilience index as a function of the amount of resource allocated to each dam. The horizontal axis shows the resources allocated to each dam separately. The two curves show the possible resilience index improvements for both dams.
Fig.~\ref{fig:13} shows that the first dam has a higher initial resilience index, however, theoretically, both dams can reach their maximum resilience index.
The values of the resilience indices for both dams, achieved at the optimal resource allocation, are marked with black squares.
From Fig.~\ref{fig:13}, we can see that the first dam's resilience index at the optimal allocation equals $0.85$, while the second dam achieves a resilience index of $0.5$. This is due to the fact that the first dam has a higher initial resilience index and it is allocated more resources. Fig.~\ref{fig:13} shows that the first dam achieves about $70\%$ increase over its initial resilience index, while the second dam achieves about $50\%$ increase over its initial resilience index. This makes the average increase of the resilience index in the system about $60\%$.

\begin{figure}[t]
  \centerline{\includegraphics[width=7cm]{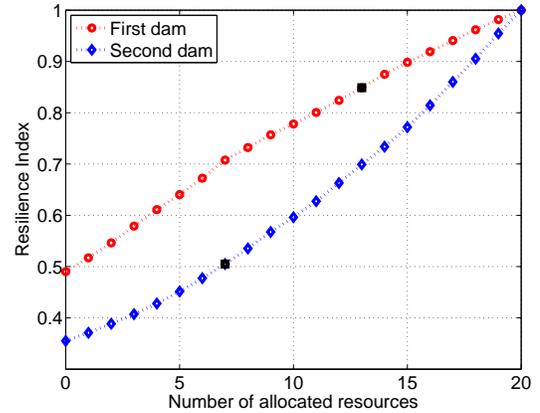}}
  \caption{Resilience Index for each dam as a relation in the amount of allocated resources to each dam.}\label{fig:13}
  \vspace{-.2cm}
\end{figure}

We next study the effect of varying the reward that the principal charges for a unit of resources on the optimal solution. We apply the same parameters as the previous experiments but the reward per resources is now varied from $\$100$ to $\$800$.

Fig.~\ref{fig:14} shows the principal's utility when applying the proposed allocation and its utility when allocating resources to only one of the dams. We see that the principal can achieve its highest utility at the value of $\$700$ per a resources unit. On the other hand, the value of $\$100$ is shown not to be enough for the principal to achieve a positive utility. The values in the range $\left[\$200-\$400\right]$ yield a negative utility for the second dam, therefore the optimal allocation is to allocate the maximum amount to the first dam. In the range $\left[\$400-\$700\right]$, both dams can achieve positive utilities and, hence, the solution involves both dams in the process of resources allocation. At the value of $\$800$, the first dam will achieve negative utility so resources are allocated to the second dam only, i.e., its maximum allowed value. From Fig.~\ref{fig:14}, we can also see that the principal can achieve a higher utility if it allocated all the resources to the first dam for reward values of $\$400$ and $\$500$. This is because the proposed solution is primarily centered around improving the resilience index and not maximizing the principal's reward. Hence, it allocates all of the available resources, as long as the principal achieves a positive utility. From Fig.~\ref{fig:14}, we can see that the proposed solution allocates $18$ and $12$ to the first dam for the reward values of $\$400$ and $\$500$, respectively. The remaining resources, i.e., $2$ and $8$ respectively are allocated to the second dam although they caused the principal's utility to be lower.

\begin{figure}[t]
  \centerline{\includegraphics[width=7cm]{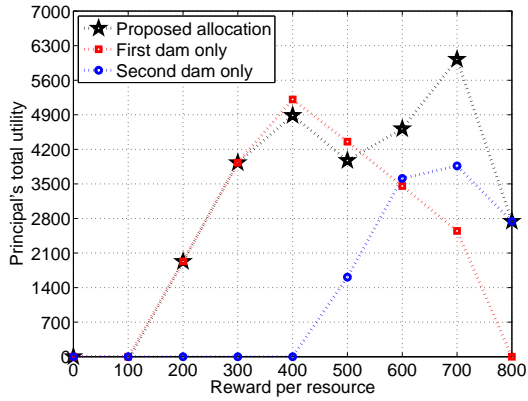}}
  \vspace*{0.1cm}
  \caption{Principal's utility as a relation in the reward charged per unit of resources.}\label{fig:14}
  \vspace{-0.3cm}
\end{figure}

In Fig.~\ref{fig:15}, for comparison purposes, we introduce a slight modification to our dynamic programming procedure to find the optimized solution from the rewards point of view. 
The main difference between the reward-optimized allocation and our original proposed allocation is that the principal does not have to allocate all the available resources. Instead, the principal allocates resources up to the limit that keeps its utility increasing. For instance, at a reward value of $\$400$, the reward-optimized allocation assigns $18$ units of resources to the first dam and nothing to the second dam, compared to $18$ and $2$ in the original proposed allocation. This helps the principal to achieve a higher utility at $\$400$ and $\$500$ as shown in Fig.~\ref{fig:15}. This reward-optimized solution coincides with the first dam's single allocation in Fig.~\ref{fig:14} for the same rewards range.
\begin{figure}[t]
  \centerline{\includegraphics[width=7cm]{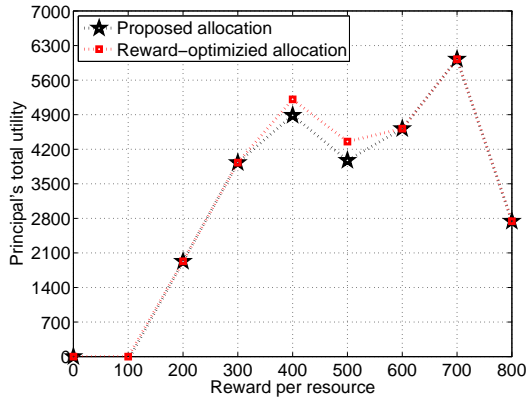}}
  \caption{Principal's utility under the proposed allocation and the introduced reward-optimized allocation.}\label{fig:15}
  \vspace{-.3cm}
\end{figure}

The extra reward achieved using the reward-optimized allocation comes at the cost of the dams' resilience. Fig.~\ref{fig:18} shows the average resilience index for both dams when using the two allocations. We see that at reward values of $\$400$ and $\$500$, the average resilience index of the reward-optimized allocation is $3\%$ and $13\%$ less than our proposed allocation, respectively. This is because less resources are used and, hence, dams can achieve less resilience improvement.
\begin{figure}[t]
  \centerline{\includegraphics[width=7cm]{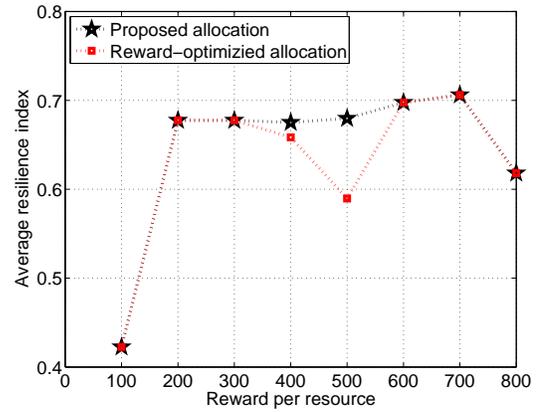}}
  \caption{Average resilience index under the proposed allocation and the introduced reward-optimized allocation.}\label{fig:18}
    \vspace{-.1cm}
\end{figure}

Finally, we show the average resilience index multiplied by the principal's utility to show the combined effect of both, we call this \emph{the average resilience utility} for the principal. Fig.~\ref{fig:16} shows that the gap between the proposed allocation and the reward-optimized allocation is smaller than the case of comparing just the utilities. This happens as the proposed mechanism allocates all the available resources which helps increase dams' resilience indices and hence the average. In the reward-optimized allocation, some resources are not allocated if they will cause the principal's utility to go lower, hence, dams achieve lower resilience indices and the average will be lower. From Fig.~\ref{fig:16}, we can see that the reward-optimized allocation slightly outperforms the proposed-allocation in the average resilience utility only at the value of $\$400$. It is slightly lower at the value of $\$500$ and coincides with the proposed allocation at all the other values. It is also clear from Fig.~\ref{fig:16} that our proposed allocation outperforms allocating resources to only one of the dams in terms of the combined effect of principal's reward and dams' resilience.

\begin{figure}[t]
  \centerline{\includegraphics[width=7cm]{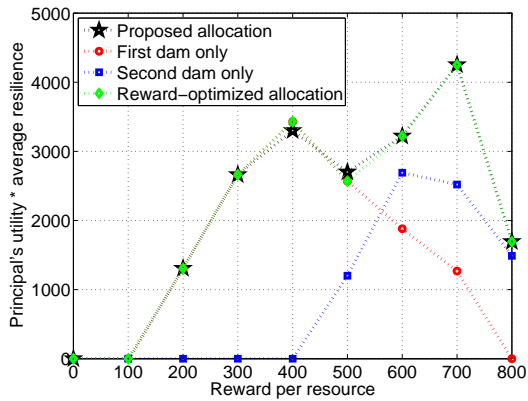}}
  \caption{The average resilience utility for the principal with the rewards value.}\label{fig:16}
    \vspace{-.6cm}
\end{figure}

Here, we note that finding the solution of the reward-optimized allocation increases the complexity of the dynamic programming optimization problem to $O\big(N \cdot (R_{\textrm{max}} - R_{\textrm{min}})^N\big)$ as the principal needs to check all partial resource allocations. This significantly increases the solution space and the time needed to reach the optimal solution. Moreover, this allocation will achieve a lower average improvement in the resilience index as less resources are allocated. Note that, the last constraint in (\ref{eq:opt}) needs to be relaxed to $\sum_{i=1}^{N}R_i \le R$ to allow for partial allocations. Given the complexity of the reward-optimized allocation and the limited improvement it can achieve over our proposed allocation, the proposed allocation will be superior in allocating the resources to a system of multiple CIs.

\section{Conclusion}\label{Sec:conclusion}

In this paper, we have proposed a novel framework to study and optimize the resilience of CIs. A novel resilience index has been introduced that is derived from a Markov chain representing the infrastructure's performance state. The state is defined to be either success, warning, or failure. The framework focuses on the effect of the probability of transition from warning to failure on the resilience index. We have then proposed a Bayesian network to model the infrastructure's physical components and their effect on the resilience index. To prioritize the infrastructure's components in the resilience improvement process, we have introduced a Bayesian network algorithm that captures the effect of each component on the infrastructure's probability of failure.
We have evaluated the proposed framework in a case study of hydropower dams. We have defined a problem of allocating resources to a system of multiple CIs and studied it within the context of the case study. The problem is modeled using contract theory in which a system operator wants to maximize the economic benefit from allocating the resources to CIs. Dynamic programming optimization has been used to derive the optimal solution for the problem of resource allocation. Results have shown that the proposed framework outperforms other allocation methods both in the economic reward for the system operator as well as the average resilience utility.
\vspace{-0.1cm}

\bstctlcite{IEEEexample:BSTcontrol}
\def\baselinestretch{1.00}
\bibliographystyle{IEEEtran}
\bibliography{references}
\vspace{-0.5cm}
\end{document}